\documentclass[twoside]{article}
\usepackage[margin=1.28in]{geometry}
\usepackage[english]{babel}
\usepackage[utf8]{inputenc}
\usepackage{graphicx}
\usepackage{amssymb}
\usepackage{amsthm}
\usepackage{tikz-cd}
\usepackage{mathrsfs}
\usepackage[colorinlistoftodos]{todonotes}
\usepackage{enumitem}
\usepackage{yfonts}
\usepackage{ dsfont }
\DeclareUnicodeCharacter{2212}{-}
\usepackage[hidelinks]{hyperref}
\usepackage{etoolbox}
\usepackage{dynkin-diagrams}
\usepackage{setspace}
\usepackage[toc,page]{appendix}
\usepackage{musicography}
\usepackage{url}
\usepackage{fancyhdr}
\usepackage{hyperref}
\usepackage{alphabeta}
\usepackage{stmaryrd}

\newcounter{theorems}[section]

\newtheorem{theorem}[theorems]{Theorem}
\newtheorem{definition}[theorems]{Definition}
\newtheorem{lemma}[theorems]{Lemma}
\newtheorem{corollary}[theorems]{Corollary}

\newtheorem{proposition}[theorems]{Proposition}
\newtheorem{claim}[theorems]{Claim}
\newtheorem{remark}{Remark}

\title{}

\author{Xenia Dimitrakopoulou}

\newcommand{\R}{\mathbb{R}}
\newcommand{\C}{\mathbb{C}}
\newcommand{\Z}{\mathbf{Z}}

\newcommand{\Q}{\mathbb{Q}}

\newcommand{\A}{\mathbb{A}}

\newcommand{\cX}{\mathcal{X}}

\newcommand{\cD}{\mathcal{D}}

\newcommand{\Qp}{\mathbf{Q}_p}
\newcommand{\Zp}{\mathbf{Z}_p}
\newcommand{\cA}{\mathcal{A}}
\newcommand{\cW}{\mathcal{W}}
\newcommand{\la}{\lambda}

\newcommand{\al}{\alpha}
\newcommand{\be}{\beta}
\newcommand{\ml}{(\mu,\la)}
\newcommand{\mlj}{(\mu,\la+j)}
\newcommand{\cO}{\mathcal{O}}
\newcommand\bigA{\makebox(0,0){{\huge{A}}}}

\newcommand{\cL}{\mathcal{L}}
\newcommand{\cP}{\mathscr{P}}
\newcommand{\cM}{\mathcal{M}}
\newcommand{\He}{\mathbf{\mathcal{H}}}
\newcommand{\fm}{\mathfrak{m}}
\newcommand{\ps}{\left(\pi,\sigma\right)}
\newcommand{\psp}{\left(\pi_p,\sigma_p\right)}
\newcommand{\Un}{U_n}
\newcommand{\Uno}{U_{n+1}}

\newcommand{\Gac}{\Gamma^{\text{ac}}}
\newcommand{\PO}{\Phi_{\Omega}}

\newcommand{\cml}{\operatorname{Crit}\left(\mu,\la\right)}
\newcommand{\IwG}{\mathrm{Iw_G}}
\newcommand{\Fp}{\mathbb{F}_p}
\newcommand{\XG}{X(\Gac)}

\normalsize

\date{}

\begin{document}

	\title{Anticyclotomic $p$-adic $L$-functions for Coleman families of $\Uno \times \Un$}
	\maketitle
	\pagestyle{fancy}
	\fancyhead{}
	\fancyhead[RE]{Xenia Dimitrakopoulou}
	\fancyhead[LO]{Anticyclotomic $p$-adic $L$-functions for Coleman families of $\Uno \times \Un$}

	\begin{abstract}
		By $p$-adically interpolating the branching law for the spherical pair $\left(U_n, U_{n+1} \times U_{n}\right)$ of definite unitary groups, we construct a $p$-adic $L$-function attached to cohomological automorphic representations of $U_{n+1} \times U_{n}$. Under a further multiplicity one assumption, we extend the construction to Coleman families. Our $p$-adic $L$-function interpolates the square root of the central critical $L$-value. It has weight and anticyclotomic variables and its construction relies on the proof of the unitary Gan--Gross--Prasad conjecture.
	\end{abstract}

	\tableofcontents
	
	\section{Introduction}

    $L$-functions are rich and intricate complex analytic functions that encode deep arithmetic information about number-theoretic objects. The study of their critical values has been a long on-going project in the area as they are expected to relate to algebraic invariants, as predicted by the deep conjectures of Birch--Swinnerton-Dyer and Bloch--Kato. The most fertile methods of proof of the partial results we have towards these open problems have included $p$-adic versions of these conjectures, for a fixed prime $p$. They predict that the $p$-adic analytic rank associated to a $p$-adic $L$-function is equal to the algebraic rank of a $p$-adic algebraic invariant. This is conditional to the existence of a $p$-adic $L$-function, which we expect can be constructed as the $p$-adic avatar of any automorphic $L$-function that has critical values.

	Inspired by \cite{kasten} and \cite{Jan16}, we study the problem of $p$-adic interpolation of Rankin--Selberg products. We use the new method of constructing $p$-adic $L$-functions via interpolation of representation-theoretic branching laws, as in \cite{GS} and \cite{BDGJW}. We construct a multi-variable $p$-adic $L$-function for cohomological automorphic forms of $U(n+1) \times U(n)$ in the anticyclotomic direction with variation in families, by exhibiting interpolation properties of classical central $L$-values.
	
	More precisely, let $E$ be an imaginary quadratic field. Fix a prime $p$ which splits in $E$. Denote by $\Gamma^{\text{ac}}$ the quotient of the idele class group giving anticyclotomic extensions of $E$ unramified outside of $p$. Let $V$ be a Hermitian space of dimension $n+1$ over $E$ and $W$ a non-degenerate subspace of codimension one. Denote the unitary groups by $\Uno=U(V),$ $\Un=U(W)$ respectively and let $G\coloneqq\Uno \times \Un$. We view $H\coloneqq U_n$ as a subgroup of $G$ via the diagonal embedding. We assume that $G(\R)$ is compact. Let $\pi=\otimes \pi_v$ (resp. $\sigma=\otimes \sigma_v$) be an irreducible cuspidal automorphic representation of $\Uno$ (resp. $\Un$) of cohomological weight $\mu=(\mu_1,...,\mu_{n+1})$ (resp. $\la=(\la_1,...,\la_n)$) such that $\operatorname{Hom}_{H_v}(\pi_v\times\sigma_v,\C)\neq 0$ for all places $v$. We further assume that $\ps$ is Iwahori spherical at $p$. We say that such a $\ps$ satisfies ($\bigstar$). Let $\pi_{n+1, E}$ (resp. $\pi_{n, E}$) be the base change of $\pi$ (resp. $\sigma$) to $\mathrm{GL}_{n+1}\left(\mathbb{A}_E\right)$ (resp. $\mathrm{GL}_{n}\left(\mathbb{A}_E\right)$). Denote by $L\left(s, \pi\times \sigma\right)$ the Rankin-Selberg convolution $L$-function $L\left(s, \pi_{n+1, E} \times \pi_{n, E}\right)$ as studied by Jacquet--Piatetski-Shapiro--Shalika \cite{JPSS}.
	
	We study the problem of $p$-adic interpolation of the central $L$-value of $L(s, \pi \times \sigma)$ at the split prime $p$. If
	
	$$-\mu_{n+1} \geq \lambda_{1} \geq-\mu_{n-1} \geq \lambda_{2} \geq \cdots \geq \lambda_{n} \geq-\mu_{1},$$
	
	then the central value $s=1/2$ of $L(s, \pi \times \sigma)$ is critical. To address this problem, we consider the branching law from $U_{n+1}$ to $U_n$, that is the set $\operatorname{Hom}_{U_{n}}\{\mathcal{V}_{\check{\mu}},\mathcal{V}_{\la}\}$ and its dimension, where  $\mathcal{V}_{\la}$ is the algebraic $U_{n}$-representation of highest weight $\la$ and $\mathcal{V}_{\check{\mu}}$ the algebraic $\Uno$-representation of highest weight $\check{\mu}=(-\mu_{n+1},...,-\mu_{1})$. 
 
    More specifically, the branching law for $\Un \subset \Uno$ is multiplicity free and states that
	\begin{equation}
		\operatorname{dim} \operatorname{Hom}_{U_{n}}\{\mathcal{V}_{\check{\mu}},\mathcal{V}_{\la}\} =1 \iff	-\mu_{n+1} \geq \lambda_{1} \geq-\mu_{n-1} \geq \lambda_{2} \geq \cdots \geq \lambda_{n} \geq-\mu_{1}. \nonumber
	\end{equation}
 
	In Section \ref{bl}, we assume that the above inequality holds and select an appropriate vector $ev_{u_{(\mu,\la)}} \in \operatorname{Hom}_{U_{n}}\{\mathcal{V}_{\check{\mu}},\mathcal{V}_{\la}\}=\operatorname{Hom}_{H}\left(\mathcal{V}_{\mu}^{\vee} \otimes \mathcal{V}_{\lambda}^{\vee}, L\right),$ for $L/\Qp$ a large enough finite extension. We consider the space of distributions $\cD_{\ml}$, which gives a $p$-adic analog of $\mathcal{V}_{\mu}^{\vee} \otimes \mathcal{V}_{\lambda}^{\vee}=\mathcal{V}_{\ml}^{\vee}$ and we $p$-adically interpolate the aforementioned branching law by constructing a map  $ \kappa_{\ml} : \cD_{\ml} \to \cD(\Gamma^{\text{ac}}, L)$ on the level of distributions such that the following diagram commutes (see Proposition \ref{blprop}):

    \begin{center}
			\begin{tikzcd}

				\cD_{\ml}(L) \arrow{d} \arrow{r}{\kappa_{\ml}}	& \cD(\Gamma^{\text{ac}},L) \arrow{d} \\
				\mathcal{V}_{\ml}^{\vee}(L) \arrow{r}{ev_{u_{(\mu,\la)}}}	& L.
			\end{tikzcd}
			
		\end{center}

	Consider now an element of Gross' style modular forms
	\begin{align*}
		f \in M(G, \mathcal{V}_{\mu}^{\vee} \otimes \mathcal{V}_{\lambda}^{\vee})=&\{f: G(\A) \to \mathcal{V}_{\mu}^{\vee} \otimes \mathcal{V}_{\lambda}^{\vee}: f \text{ is locally constant on } G(\A)  \\
		&\text{ and } f(\gamma g)=\gamma f(g) \text{ for all } \gamma \in G(\Q)\}.
	\end{align*}
	By composing it with  $ev_{u_{(\mu,\la)}}$, we get an automorphic form $\phi =  ev_{u_{(\mu,\la)}} \circ f \in \pi \times \sigma$. We can therefore compute its automorphic period integral 
 
                    $$\cP_H(\phi)=\int_{Z^H(\mathbb{A}) H(\Q) \backslash H(\mathbb{A})} \phi(h) d h .$$
 
 and apply the unitary Gan--Gross--Prasad conjecture (now a theorem; see \cite{ggp}, \cite{BPLZZ}, \cite{GGPfinal} and Section \ref{ggp}) to get 
	
	$$\frac{|\cP_{H}(\phi)|^2}{\langle \phi, \phi^{\vee} \rangle}=\frac{\Delta}{S_{\pi\times\sigma}} \frac{L(1/2, \pi\times\sigma)}{L(1,\pi,\mathrm{Ad})L(1,\sigma,\mathrm{Ad})}\prod_{\nu \in {S}}\alpha_{\nu}(\phi_{\nu},\phi_{\nu}),$$
	
	where $\Delta$ is a certain product of abelian $L$-values, $S_{\pi\times\sigma}$ is the size of the Arthur parameter of $\pi\times\sigma$, $S$ is a large enough set of primes containing $p$ and $\alpha_{\nu}(\phi_{\nu},\phi_{\nu})$ is a regularized matrix coefficient.
	
	Fix $\widetilde{\ps}=(\ps, a_p)$ a non-critical $p$-refinement of ${\ps}$ (this is a choice of eigenvalue $a_p$ with a specific slope condition as explained in Section \ref{HA}).	We would like to $p$-adically interpolate twists by anticyclotomic accessible characters of $\ps$. These are characters of $\Gac$ of infinity type $(-j,j)$, for $j$ in the set	
	$$\cml=  \{j \in \Z | -\mu_{n+1} \geq \la_{1} +j \geq -\mu_{n} \geq \la_{2}+j \geq...\geq \la_{n}+j \geq -\mu_{1}\}.$$
	
	By replacing the space of algebraic modular forms with their $p$-adic analogue of distribution valued forms $M_p(G, \cD_{\ml})$ and the branching law vector $ev_{u_{(\mu,\la)}}$  with its $p$-adic avatar $\kappa_{\ml}$, for each classical form $\phi \in \widetilde{\ps}$, we construct a distribution $\cL_p(\phi) \in \cD(\Gamma^{\text{ac}}, L)$ such that integrating an accessible character $\chi \in \cX(\Gamma^{\text{ac}})$ of conductor $p^{\beta}$ with $\beta>0$ against it interpolates the square root of the central critical value $L(1/2, \pi\times(\sigma \otimes \chi)).$ 
	
	In particular, our first main result is the following (Theorem \ref{theorem1}):
	
	\begin{theorem}\label{thm1}
		There exists a distribution $\cL_p(\widetilde{\ps}) \in \cD(\Gamma^{\text{ac}}, L)$ of growth $v_p(a_p)$ such that for all accessible characters $\chi \in \cX(\Gamma^{\text{ac}})$ of conductor $p^{\beta}$ with $\beta>0$, if we set
		
		$$\cL_p(\widetilde{\ps}, \chi) \coloneqq \int_{\Gamma^{\text{ac}}} \chi(x) d\cL_{p}(\widetilde{\ps}),$$ 
		
		then 
		
		$$|\cL_p (\widetilde{\ps}, \chi)|^2=\frac{\Delta}{S_{\pi\times\sigma}} \frac{L(1/2, \pi\times\left(\sigma\otimes\chi\right))}{L(1,\pi,\mathrm{Ad})L(1,\sigma,\mathrm{Ad})} \cdot \alpha^{\star},$$
		
		where $\Delta$ is a product of abelian $L$-values, $S_{\pi\times\sigma}$ relates to the Arthur parameter of the representation $\ps$ and the correction factor $\alpha^{\star}$ depends on the choice of Haar measure, the $p$-refinement of $\ps$ and the conductor of $\chi.$ 
		
		If moreover $v_p(a_p)< \#\cml$, then $\cL_p(\widetilde{\ps})$ is uniquely determined by the above interpolation formula.
		
	\end{theorem}
	
	Our approach allows for arbitrary cohomological weight and relaxes the ordinarity condition of \cite[Thm 5.2]{Liu}, who however works over CM extensions. Our results can be adapted to the more general CM case as well, but we choose to avoid technical complications in order to ease exposition.

    Assuming that $\widetilde{\ps}$ has the multiplicity one property with respect to its level (see Definition \ref{multassumption}), we can extend the above theorem to a Coleman family over $\widetilde{\ps}.$ For an affinoid $\Omega$ around $\ml$ in the weight space, we denote the space of distributions over it by $\cD_{\Omega}$. In Proposition \ref{blprop} and Corollary \ref{blcor}, we construct a map $ \kappa_\Omega : \cD_{\Omega} \to \cD(\Gamma^{\text{ac}}, \cO_\Omega)$, such that the following diagram commutes:
        \begin{center}
			\begin{tikzcd}
				
				\cD_\Omega \arrow{d}{sp_{\ml}} \arrow{r}{{\kappa_\Omega}}
				& \cD(\Gac, \cO_\Omega) \arrow{d}{sp_{\ml}} \\
				\cD_{\ml}(L) \arrow{d} \arrow{r}{\kappa_{\ml}}	& \cD(\Gamma^{\text{ac}},L) \arrow{d} \\
				\mathcal{V}_{\ml}^{\vee}(L) \arrow{r}{ev_{u_{(\mu,\la)}}}	& L,
			\end{tikzcd}
			
		\end{center}

  where $sp_{\ml}$ is the specialisation to weight $\ml.$
	
	In Section \ref{plf}, we rewrite the automorphic period integral as 
 $$\cP_H:  M(G, \mathcal{V}_{\mu}^{\vee} \otimes \mathcal{V}_{\lambda}^{\vee}) \to L$$
	
	via $$\cP_H(\phi)= \mu(K) \sum_{x \in K\backslash G(\A)/G(\Q)}  \frac{ev_{u_{(\mu,\la)}}(\phi(x))}{\left|\Gamma \cap x^{-1}Kx\right|},$$
	
	where $K$ is the level of $\phi$. We then define the $p$-adic avatar of the integral, namely a map $\cL_p : M_p(G, \cD_{\Omega})^{a_p} \to \cD(\Gamma^{\text{ac}}, \cO_\Omega),$ on the $p$-adic modular forms with values in $\cD_{\Omega}$, localised at a non-critical $p$-refinement $a_p$ of $\ps$ given by

	$$\cL_p\left(\phi\right)=\mu(K) \sum_{x \in K\backslash G(\A)/G(\Q)}  \frac{\kappa_{\Omega}(\phi(x))}{\left|\Gamma \cap x^{-1}Kx\right|}.$$

	Using \cite{BDJ} under the further multiplicity one assumption for the level of $\ps$ (see Section \ref{fam}), for the affinoid $\Omega$ in the weight space, we show the existence of a Coleman family $\Phi_{\Omega}$ passing through $\widetilde{\ps}$. Therefore, for every classical $\phi$, which is the specialisation of $\PO$ at a classical point, we can integrate characters $\chi \in \cX(\Gamma^{\text{ac}})$ against $\cL_p(\phi)$ and thus interpolate $L(1/2, \pi\times(\sigma \otimes \chi)).$ In Section \ref{plf} we prove our second main result (Theorem \ref{Mthm}):

	\begin{theorem}\label{mainthm}
		
		There exists a unique up to scalars non-zero distribution $\cL_p(\PO) \in \cD(\Gamma^{\text{ac}}, \cO_\Omega)$ of growth $v_p(a_p),$
		such that for a classical point $\ml \in \Omega$ whose corresponding automorphic representation $\ps$ satisfies ($\bigstar$), integrating $sp_{\ml}\left(\cL_p(\Phi_{\Omega})\right)$ against an accessible character $\chi \in \cX(\Gamma^{\text{ac}})$ of conductor $p^{\beta}$ with $\beta>0$, interpolates the square root of the central critical value $L(1/2, \pi\times(\sigma \otimes \chi)).$ In particular, there exists test vector $\phi$, such that if we set
		
		$$\cL_p(\phi)=sp_{\ml}\left(\cL_p(\Phi_{\Omega})\right) \text{  and  
 }\cL_p(\widetilde{\ps}, \chi) \coloneqq \int_{\Gamma^{\text{ac}}} \chi(x) d\cL_{p}(\phi),$$ 
		
		then
		
		$$|\cL_p (\widetilde{\ps}, \chi)|^2=\frac{\Delta}{S_{\pi\times\sigma}} \frac{L(1/2, \pi\times\left(\sigma\otimes\chi\right))}{L(1,\pi,\mathrm{Ad})L(1,\sigma,\mathrm{Ad})} \cdot \alpha^{\star},$$
		
		where $\Delta$ is a product of abelian $L$-values, $S_{\pi\times\sigma}$ is the size of the Arthur parameter of the representation $\ps$ and the correction factor $\alpha^{\star}$ depends on the choice of Haar measure, the test vector, the $p$-refinement of $\ps$ and the conductor of $\chi.$ 
		
	\end{theorem}
	
	Note that the uniqueness in this case follows from the non-criticality of the $p$-refinement, unlike the previous theorem where condition on the slope needed to be more restrictive.
	
	\section*{Acknowledgements}
	I am indebted to my advisors David Loeffler and Chris Williams for suggesting the problem and for many enlightening discussions. I would also like to thank them for their comments on earlier drafts of this work as well as their guidance throughout the process. 
	
	I am thankful to Daniel Barrera Salazar for helpful discussions and the Universidad de Santiago de Chile for hosting me while part of this work was being completed.
	
	I am a Swinnerton Dyer scholar, supported by the Warwick Mathematics Institute Centre for Doctoral Training. I gratefully acknowledge funding by University of Warwick’s Chancellors' International Scholarship scheme. 
	
	\section{Preliminaries}\label{preliminaries}
	
	\subsection{Notation}\label{notation}

 Let $E$ be an imaginary quadratic field and let $\eta_{E / \Q}: \Q^{\times} \backslash \A^{\times} \to \{ \pm 1\}$ be the associated quadratic character. Fix a prime $p$ which splits in $E$ and denote by $P$ the primes of $E$ above $p$. We set:
	
	$$\Gamma_{\Q}=\Q^{\times} \backslash\left(\mathbb{A}^{\infty}\right)^{\times} / \prod_{q\neq p} \mathbb{Z}_q^{\times} \text{  and  } \Gamma=E^{\times} \backslash\left(\mathbb{A}_E^{\infty}\right)^{\times} / \prod_{\nu \notin \mathrm{P}} \cO_{E_{\nu}}^{\times}.$$
	
	Let $\Gamma^{\text{ac}}=\text{ker}\left(\text{Nm}_{E/\Q}:\Gamma \to \Gamma_{\Q}\right)$ and denote by $\XG$ the set of anticyclotomic characters.
	
	Let $V$ be a Hermitian space of dimension $n+1$ and $W$ a non-degenerate subspace of codimension one. Denote the unitary groups by $\Uno=U(V)$ and $\Un=U(W)$ respectively. We assume that $U(W)(\R)$ (resp. $U(V)(\R)$) has singature $(n,0)$ (resp. $(n+1,0)$), so that it's compact. Let $G=U(W) \times U(V)$ and $H$ be the image of $U(W)$ into $G$ via the diagonal embedding \begin{align}
		\iota :  U(W)  & \hookrightarrow G  \nonumber \\
		g & \mapsto \left(\begin{psmallmatrix}
			g &  \\
			& 1 \\
		\end{psmallmatrix}, g\right).  \nonumber 
	\end{align}

    Note that since $p$ splits in $E$, $G(\Qp)=GL_{n+1}(\Qp)\times GL_{n}(\Qp).$
	
	We denote the automorphic quotient of $H$ by $\left[H\right]= Z^{H}(\A) H(\Q) \backslash H(\A)$, where $Z^{H}(\A)$ is the intersection of the adelic points of the centre of $G$ with the adelic points of $H$.
	
	We denote by $B$ the Borel subgroup of $G$ of upper triangular matrices, $\overline{B}$ its opposite Borel subgroup of lower triangular matrices, $T$ its maximal torus and consider decompositions $B=TN$, $\overline{B}=T\overline{N}$ for $N$, $\overline{N}$ the unipotent radicals of $B$ and $\overline{B}$ respectively.
	
	Let $X(T)$ denote the set of algebraic characters of the torus, whose elements we denote by $\ml$ where $\mu=(\mu_{1},..., \mu_{n+1})$ a character of the torus of the $\Uno$-component and $\la=(\la_{1},..., \la_{n})$ is a character of the torus of the $U_{n}$-component. We call a character $\ml\in X(T)$ a weight of $G$ and we say it's dominant if 

            $$\mu_1\geq \mu_2 \geq ...\geq \mu_{n+1} \textit{ and } \la_1\geq\la_2\geq...\geq \la_n.$$
	
	Let $\pi\times \sigma$ be an irreducible cuspidal automorphic representation of $G(\mathbb{A})$ of cohomological weight $\ml$ such that $\operatorname{Hom}_{H_v}(\pi_v\times\sigma_v,\C)\neq 0$ for all places $v$. Let $\pi_{n, E}$ (resp. $\pi_{n+1, E}$) be the base change of $\sigma$ (resp. $\pi$) to $\mathrm{GL}_n\left(\mathbb{A}_E\right)$ (resp. $\mathrm{GL}_{n+1}\left(\mathbb{A}_E\right)$). Denote by $L\left(s, \pi\times \sigma \right)$ the Rankin-Selberg convolution $L$-function $L\left(s, \pi_{n, E} \times \pi_{n+1, E}\right)$ as studied by Jacquet--Piatetski-Shapiro--Shalika \cite{JPSS}.
	
	Define the Iwahori subgroup of $G$ by 
	
	$$\IwG:=\left\{g \in \mathrm{G}\left(\Zp\right): g\left(\bmod p\right) \in B\left(\Fp\right)\right\}.$$
	
	Throughout the paper, we work with Iwahori level $K$ and assume $\ps$ is Iwahori-spherical at $p$, i.e. $\psp^{\IwG} \neq 0.$

     We say that a $\ps$ as in this section satisfies ($\bigstar$).
	\subsection{Hecke Algebra}\label{HA}
	Consider the operator at $p$ given by $U_{p}=\left[\IwG t_{p} \IwG\right]$, where $$t_{p}=\left(\text{diag}(p^n,p^{n-1},...,p,1), \text{diag}(p^n,p^{n-1},...,p)\right) \in G.$$ 
	
	Define the Hecke algebra at $p$ as $\He_p= \Zp \left[U_{p}\right].$
	
	If $\nu$ is a dominant integral weight in $X(T)$ and $l\neq p$ is a prime, let $T_{\nu,l}=\left[G(\mathbf{Z}_l)\nu(l)G(\mathbf{Z}_l)\right]$. The Hecke algebra away from $p$ is defined as $$\He'=\Z\left[T_{\nu,l}| \text{ for all primes }l \neq p \text{ and all dominant integral weights } \nu\right].$$ The full Hecke algebra is $\He=\He'\otimes\He_p.$

    Let $F$ be a number field containing the Hecke field of $\pi_f$. Attached to $\pi$ we have a homomorphism
$$
\psi_\pi: \mathcal{H}^{\prime} \otimes F \rightarrow F
$$
which for $\nu \in X(T)$ and $l\neq p$ sends $T_{\nu,l}$ to its eigenvalue acting on the line $\pi_v^{K_v}$. Let $\mathfrak{m}_\pi:=\operatorname{ker}\left(\psi_\pi\right)$, a maximal ideal in $\mathcal{H}^{\prime} \otimes F$. If $L$ is any field containing $F$, we get an induced maximal ideal in $\mathcal{H}^{\prime} \otimes L$, which by abuse of notation we also denote $\mathfrak{m}_\pi$.

If $M$ is a finite-dimensional $L$-vector space with an action of $\mathcal{H}^{\prime}$, then we denote by the localisation $M_{\mathfrak{m}_\pi}$ the generalised eigenspace $M \llbracket \mathfrak{m}_\pi \rrbracket $ attached to $\psi_\pi$.

	Finally, we will need the following definitions:
	
	\begin{definition}
		Suppose $\psp$ is Iwahori-spherical. A $p$-refinement for the $p$-part $\psp$ of an automorphic representation $\ps$ of $G$ is a choice of a system of Hecke eigenvalues on each component appearing in the decomposition of $\psp^{\IwG}$, that is a choice $\widetilde{\ps} = (\ps, \alpha_p)$, where $\alpha_p=\beta_p\gamma_p$ is a system of $\He_p$-eigenvalues appearing in $\psp^{\IwG}$. Equivalently, there exists an eigenvector $\left(\phi, \varphi\right) \in \psp^{\IwG}$ such that $U_{p}\left(\phi, \varphi\right)= \left(\beta_{p}\phi, \gamma_{p}\varphi\right)=\alpha_{p}\left(\phi, \varphi\right).$
	\end{definition}

    \begin{definition}
         Let $p$-refinement $\widetilde{\ps}$ with $U_{p}$-eigenvalue $\alpha_{p}=\beta_{p}\gamma_{p}$. Define integral normalisations
         \begin{itemize}
             \item $U_p^{\circ}\coloneqq \ml(t_p)U_p$,
             \item $\alpha_{p}^{\circ}=\ml(t_p)\alpha_{p},$
             \item $\beta_{p}^{\circ}=\mu(\text{diag}(p^n,p^{n-1},...,p,1))\beta_{p}$ and
             \item $\gamma_{p}^{\circ}=\la(\text{diag}(p^n,p^{n-1},...,p))\gamma_{p}.$
         \end{itemize}
    \end{definition}

    As explained in \cite[Remark 3.23]{BW21},  $\beta_{p}^{\circ}$ and $\gamma_{p}^{\circ}$ are $p$-integral.
	
	\begin{definition}\label{crit}
		A $p$-refinement $\widetilde{\ps}$ with $U_{p}$-eigenvalue $\alpha_{p}=\beta_{p}\gamma_{p}$ is of non-critical slope if $v_p\left(\beta_{p}^{\circ}\right)<\min_{i}\left(\mu_{i}-\mu_{i+1}+1\right)$ and $v_p\left(\gamma_{p}^{\circ}\right)<\min_{j}\left(\la_{j}-\la_{j+1}+1\right).$
	\end{definition}
	
	As we will see in Section \ref{fam}, this is the small slope condition we will use to exhibit the existence of a Coleman family, following \cite[\S 4.1]{BW21}.

	\subsection{Branching Laws}\label{slopes}

	Let $\mathcal{V}_{\la}$ be a $U_{n}$-representation of highest weight $\la$ and $\mathcal{V}_{\mu}$ a $\Uno$-representation of highest weight $\mu$. The branching law from $\Uno$ to $\Un$ states that:
	
	\begin{theorem}\label{brl}
		$\operatorname{Hom}_{U_{n}}\left(\mathcal{V}_{\mu}^{\vee} \otimes \mathcal{V}_{\lambda}^{\vee}, L\right) \neq 0$ if and only if 
		
		$$-\mu_{n+1} \geq \la_1 \geq -\mu_{n}\geq...\geq -\mu_2 \geq \la_n \geq -\mu_{1}.$$
		
		If that is the case, $\operatorname{Hom}_{U_{n}}\left(\mathcal{V}_{\mu}^{\vee} \otimes \mathcal{V}_{\lambda}^{\vee}, L\right)$ is a one-dimensional space.
		
	\end{theorem}
	\begin{proof}
		See \cite[Theorem 8.1.1]{symmetry} and see also \cite[Corollary 2.20]{raghuram}.
	\end{proof}
	
	When this inequality between $\mu$ and $\la$ holds, we say $\mu$ interlaces $\la.$
	
	We can therefore study the set
	
	$$\cml=  \{j \in \Z | -\mu_{n+1} \geq \la_{1} +j \geq -\mu_{n} \geq \la_{2}+j \geq...\geq \la_{n}+j \geq -\mu_{1}\},$$
	
	or equivalently the integers $j$ such that
	
	$$ \min_{i}\left(\la_i-\mu_{n+2-i}\right) \geq j \geq \max_{i}\left(-\mu_{n+1-i}-\la_{i}\right).$$
	
	Set $h_{\ml}\coloneqq|\min_{i}\left(\la_i-\mu_{n+2-i}\right)-\max_{i}\left(-\mu_{n+1-i}-\la_{i}\right)|$.
	
	\begin{definition}
		We say that $\widetilde{\ps}$ is of very small slope if $v_p(\alpha_{p})<h_{\ml}.$
	\end{definition}
	
	We note that if $\widetilde{\ps}$ is of very small slope, then it is of non-critical slope. 
	
	For each $j \in \cml$, we can consider an anticyclotomic character $\chi \in \XG$ of infinity type $(-j,j)$. We then consider the twist of $\ps$ by $\chi$ as we want to study the central $L$-value of $L(s, \pi \times (\sigma\otimes \chi))$, which is critical, see \cite[Chapter 10]{LinThesis}.
	
	\begin{definition}
		We call a character $\chi \in \XG$ accessible for $\ps$ if it is of infinity type $(-j,j)$ for some $j \in \cml$.
	\end{definition}

	
	
	\subsection{Distributions}\label{distributions}
We introduce the notion of growth which will be crucial in Proposition \ref{growth} for the uniqueness of our $p$-adic $L$-function.

For a finite $\Qp$-extension $L$, we let $\cA(\Zp,L)[r]$ be the space of locally analytic functions on $b+p^r\Zp$ for all $b \in \Zp$ and $\cA(\Zp,L)=\bigcup_{r}\cA(\Zp,L)[r].$ We denote by $\cD(\Zp,L)[r]$ the dual of $\cA(\Zp,L)[r]$ and by $\cD(\Zp,L)=\bigcap_{r}\cD(\Zp,L)[r]$ the dual of $\cA(\Zp,L).$

    The space of distributions $\cD(\Zp,L)[r]$ is naturally a Banach space and we denote its norm  by $\lvert\cdot\rvert_r.$

    \begin{definition}
        We say that a distribution $\mu \in \cD(\Zp,L)$ has growth $h \in \Q_{\geq0}$ if there exists a constant $C\geq0 $ such that for all $r \in \Z_{\geq 0}$ we have $\lvert\mu\rvert_r\leq Cp^{rh}.$
    \end{definition}

    We denote the space of distributions of growth $h$ by $\cD^h(\Zp, L),$
    
     $$ \cD(\Zp, L) \supset \cD^1(\Zp, L)\supset ...\supset \cD^h(\Zp, L)\supset...\supset \cM(\Zp, L),$$

    where $\cM(\Zp, L)=\cD^0(\Zp, L)$ is the space of measures.

	\section{\texorpdfstring{$p$}{p}-adic Interpolation of Branching Laws}\label{bl}
	
	This section aims to $p$-adically interpolate the branching law of Theorem \ref{brl}.
		\subsection{\texorpdfstring{$p$}{p}-adic interpolation}
 
 Take $\ml \in X(T)$ and $\mathcal{V}_{\ml}$ the algebraic 
	representation of $G$ of highest weight $\ml$. From now on, fix an extension $L$ of $\Qp$. 

 The following is known as the Borel--Weil--Bott Theorem \cite[Theorem 12.1.6]{symmetry}:
	\begin{theorem}
		$\mathcal{V}_{\ml}(L)$ can be explicitly described as the algebraic induction 
		\begin{align}
			\operatorname{Ind}_{{B}\left(\mathbf{Q}_{p}\right)}^{G\left(\mathbf{Q}_{p}\right)}& \ml (L)=\{ f: G\left(\mathbf{Q}_{p}\right) \rightarrow L : f \text{ is algebraic and }\nonumber\\
			&f\left(\overline{n} t g\right)=\ml(t) f(g) \quad \forall \overline{n}\in \overline{N} \left(\mathbf{Q}_{p}\right), t \in T\left(\mathbf{Q}_{p}\right), g \in G\left(\mathbf{Q}_{p}\right) \}. \nonumber
		\end{align}	
	\end{theorem}
	
	The action of $\gamma \in G\left(\mathbf{Q}_{p}\right)$ on $f \in \mathcal{V}_{\ml}(L)$ is by $(\gamma \cdot f)(g):=f(g \gamma)$.
	As $G\left(\mathbf{Z}_{p}\right)$ is Zariski-dense in $G\left(\mathbf{Q}_{p}\right)$, we can identify $\mathcal{V}_{\ml}(L)$ with the set of algebraic functions $f: G\left(\mathbf{Z}_{p}\right) \rightarrow$ $L$ satisfying 
	\begin{equation}
		f\left(\overline{n} t g\right)=\ml(t) f(g) \quad \forall \overline{n}\in \overline{N} \left(\mathbf{Z}_{p}\right), t \in T\left(\mathbf{Z}_{p}\right), g \in G\left(\mathbf{Z}_{p}\right). \label{a}
	\end{equation}
	
	Consider the following weights for $U_{n}$: $$\al_0=(0,0,...,0), \al_{1}=(1,0,...,0),...,\al_n=(1,...,1).$$

    Similarly, consider the following weights for $U_{n+1}$:

    $$\be_1=(1,0,...,0), \be_{2}=(1,1,0,...,0),...,\be_{n+1}=(1,...,1).$$

    The weights $(\be_1,\al_0), (\be_1,\al_1),..., (\be_n,\al_n), (\be_{n+1},\al_n)$ form a basis for the weight space of $G$ as any weight $\ml$ can be written as:
    \begin{equation}
    \begin{aligned}\label{interlace1}
    \ml=(\mu_1-\la_1)(\be_1,\al_0)+(\la_1-\mu_2)(\be_1,\al_1)+....+\\
    +(\mu_n-\la_n)(\be_n,\al_n)+\mu_{n+1}(\be_{n+1},\al_n).
    \end{aligned}
    \end{equation}

	We have that $\be_i$ interlaces $\al_i$ for all $1 \leq i \leq n$ and $\be_{i+1}$ interlaces $\al_{i}$ for all $0 \leq i \leq n$. Therefore, Theorem \ref{brl} for $H\subset G$ implies that the following subspaces contain the trivial representation with multiplicity one:
	
	\[\mathcal{V}_{(\be_{i+1}, \al_i)}\mid_H \text{ for } 0 \leq i \leq n\]
 and 
 \[\mathcal{V}_{(\be_{i}, \al_i)}\mid_H, \text{ for } 1 \leq i \leq n \]
	
	Hence, we have generators for the trivial representation, which we denote by

\[u_{i+1,i} \in \mathcal{V}_{(\be_{i+1}, \al_i)}\mid_H \text{ for } 0 \leq i \leq n\]
 and 
 \[v_{i,i} \in \mathcal{V}_{(\be_{i}, \al_i)}\mid_H, \text{ for } 1 \leq i \leq n .\]

	We view the elements $u_{i,i}, v_{i+1,i}$ as algebraic functions $G(\Zp) \to \Qp$.

    From now on, assume that $\mu$ interlaces $\lambda$.

	\begin{lemma}
		Let $(\mu, \la)$ be a dominant algebraic weight for $G$ such that $\mu^{\vee}$ interlaces $\la$. Then 	
$$u_{(\mu,\la)}\coloneqq \prod_{i=0}^{n-1}u_{i+1,i}^{(\mu_i-\la_i)}  \prod_{i=1}^{n}v_{i,i}^{(\la_i-\mu_{i+1})} \cdot u_{n+1,n}^{\mu_{n+1}}$$
is a generator for the trivial representation inside $\mathcal{V}_{(\mu,\la)}(\Qp)\mid_{H(\Zp)}.$
		\end{lemma}
	
	\begin{proof} 
     Firstly, we note that since $\mu$ interlaces $\la$, we have $\mu_i-\la_i\geq 0$ and $\la_i-\mu_{i+1}\geq 0,$ for all $i$. Therefore, each function in the product is algebraic and so $u_{(\mu,\la)}$ is also algebraic. 

    By \eqref{interlace1}, if $t \in T(\Zp)$ we have:
        \begin{align*}            
        \ml(t)=(\be_1,\al_0)(t)^{(\mu_1-\la_1)} (\be_1,\al_1)(t)^{(\la_1-\mu_2)}\cdot ... \cdot\\
    (\be_n,\al_n)(t)^{(\mu_n-\la_n)} \cdot (\be_{n+1},\al_n)(t)^{\mu_{n+1}}.\end{align*}
        
   Taking $\bar{n} \in \bar{N}(\Zp)$ and $g \in G(\Zp)$, we get $u_{(\mu,\la)}(\bar{n}tg)=\ml(t)u_{(\mu,\la)}(g)$ and so $u_{(\mu,\la)} \in \mathcal{V}_{(\mu,\la)}\mid_{H(\Zp)}.$

   The claim follows by noting that $H$ acts trivially on each $u_{i,i}, v_{i+1,i}$ and therefore acts trivially on $u_{(\mu,\la)}.$
	\end{proof}
	
	It is now clear that if $j \in \cml$, then if we write $\la+j$ for $(\la_1+j,...,\la_n+j)$ then:
\begin{align*}
   u_{(\mu,\la+j)}= \prod_{i=0}^{n-1}u_{i+1,i}^{(\mu_i-\la_i-j)}  \prod_{i=1}^{n}v_{i,i}^{(\la_i+j-\mu_{i+1})} \cdot u_{n+1,n}^{\mu_{n+1}} 
   \end{align*}

is a generator for the trivial representation inside $\mathcal{V}_{(\mu,\la+j)}(\Qp)\mid_{H(\Zp)}.$
	
	\subsection{Support Conditions}
	
	For an $\left(n+1\right) \times \left(n+1\right)$-matrix $g$, we denote by $\left[g\right]_n$ its top left $n \times n$ entries.
	
	Let $$N^1(\Zp)=\{(g, g') \in N(\Zp) : \left[g\right]_n=g'\text{ and } (g, g')\equiv  \left(\begin{psmallmatrix}
		1 & 0 & \cdots &0 &1 \\
		0 & 1 & \cdots &0 &1 \\
		\vdots  & \vdots  &  & \ddots &\vdots \\
		0 & 0 & \cdots & 0 &1 \\
	\end{psmallmatrix} , 1_n \right)\mod p 
	\}$$
	
	and for an integer $\beta>0$
	
	$$N^{\beta}(\Zp)=\Bigl\{(g, g') \in N(\Zp) : \left[g\right]_n=g'\text{ and } (g, g')\equiv  \left(\begin{psmallmatrix}
		1 & 0 & \cdots &0 &1 \\
		0 & 1 & \cdots &0 &1 \\
		\vdots  & \vdots  &  & \ddots &\vdots \\
		0 & 0 & \cdots & 0 &1 \\
	\end{psmallmatrix} , 1_n \right)\mod p^{\beta} 
	\Bigr\}.$$
	
	For convenience, set $g_0=\begin{psmallmatrix}
		1 & 0 & \cdots &0 &1 \\
		0 & 1 & \cdots &0 &1 \\
		\vdots  & \vdots  &  & \ddots &\vdots \\
		0 & 0 & \cdots & 0 &1 \\
	\end{psmallmatrix} .$
	
	The aim of this section is to prove the following:
	
	\begin{lemma}\label{1+pzp}
		For all $(v,u)\in \mathcal{V}_{\ml}\mid_{H(\Zp)}$, we have that up to rescaling $(v,u) \big(N^1(\Zp)\big) \subset \Zp^{\times}$.
	\end{lemma}
	\begin{proof}
		Let $(g,A) \in N^1(\Zp)$, then $g$ is of the form 
		$$g=
		\begin{pmatrix}
			&	\bigA &	&  &\begin{matrix} 1+a_1p \\ 1+a_2p \\ \vdots \end{matrix} \\ 
			0 & \cdots & &0 &1 \\
		\end{pmatrix} $$
		
		where $A$ is in the unipotent radical of $H$, such that $A \equiv 1_n \mod p $ and $a_i \in \Zp.$ 
		
		It is now easy to see that 
		$$g= \begin{psmallmatrix}
			1 & 0 & \cdots &0 &1+a_1p \\
			0 & 1 & \cdots &0 &1+a_2p \\
			\vdots  & \vdots  &  & \ddots &\vdots \\
			0 & 0 & \cdots & 0 &1 \\
		\end{psmallmatrix} 
		\begin{pmatrix}
			& \bigA &	&  &\begin{matrix} 0 \\ 0 \\ \vdots \end{matrix} \\ 
			0 & \cdots && 0 &1 \\
		\end{pmatrix}$$
		
		and that
		
		$$
		\begin{psmallmatrix}
			1 & 0 & \cdots &0 &1+a_1p \\
			0 & 1 & \cdots &0 &1+a_2p \\
			\vdots  & \vdots  &  & \ddots &\vdots \\
			0 & 0 & \cdots & 0 &1 \\
		\end{psmallmatrix}=
		\begin{psmallmatrix}
			1+a_1p & 0 & \cdots &0 &0 \\
			a_2p & 1 & \cdots &0 &0 \\
			\vdots  & \vdots  &   \ddots &\vdots &\vdots \\
			a_{n}p & 0 & \cdots & 1 &0 \\
			0 & 0 & \cdots & 0 &1 \\
		\end{psmallmatrix}
		\begin{psmallmatrix}
			1 & 0 & \cdots &0 &1\\
			0 & 1 & \cdots &0 &1 \\
			\vdots  & \vdots  &  & \ddots &\vdots \\
			0 & 0 & \cdots & 0 &1 \\
		\end{psmallmatrix}
		\begin{psmallmatrix}
			\frac{1}{1+a_1p} & 0 & \cdots &0 &0 \\
			\frac{-a_2p}{1+a_1p} & 1 & \cdots &0 &0 \\
			\vdots  & \vdots  &   \ddots &\vdots &\vdots \\
			\frac{-a_{n}p}{1+a_1p} & 0 & \cdots & 1 &0 \\
			0 & 0 & \cdots & 0 &1 \\
		\end{psmallmatrix} .
		$$
		
		Therefore, we have that
		
		$ g= \overline{b} \begin{psmallmatrix}
			1 & 0 & \cdots &0 &1 \\
			0 & 1 & \cdots &0 &1\\
			\vdots  & \vdots  &  \ddots &\vdots &\vdots \\
			0 & 0 & \cdots & 0 &1 \\
		\end{psmallmatrix} h=\overline{b}g_0h $ and $g'=  \overline{b}'1_n[h]_n,$ for $\left(\overline{b},\overline{b}'\right) \in \overline{B}(\ \Zp), \left(h,[h]_n\right) \in H(\ \Zp).$ 	
		
		Since $\left(v,u\right)\left(\left(\overline{b},\overline{b}'\right)\left(g_0,1_n\right)\left(h,[h]_n\right)\right)=\mu(\overline{b})\la(\overline{b}') \left(v,u\right)\left(g_0,1_n\right) \in \Zp^{\times}\left(v,u\right)\left(g_0,1_n\right)$ and $\overline{B}\left(g_0,1_n\right)H(\Zp) $ is open and dense in $G(\Zp),$ we can rescale by an element of $\Qp^{\times}$ to get $ \left(v,u\right)\left(g_0,1_n\right)=1$. The claim now follows.
		
	\end{proof}

	\subsection{The Interpolation Theorem}

	Consider an affinoid $\Omega$ in the weight space of $G$ around $\ml$. We allow $\Omega= \{\ml\}$. Such an affinoid is equipped with a character $\chi_\Omega=(\chi_{\Omega_1}, ..., \chi_{\Omega_{2n+1}}) :T(\Zp) \to \cO^{\times}_\Omega$ such that when composed with the evaluation at $\ml$, it equals $\ml$, i.e. $sp_{\ml} \circ \chi_\Omega= \ml$.

	\begin{definition}Let $\mathcal{A}\left(\operatorname{Iw}_{G}, \mathcal{O}_{\Omega}\right)$ be the space of locally analytic functions $f : \operatorname{Iw}_{G} \to \mathcal{O}_{\Omega}.$
		Define $\mathcal{A}_{\Omega}:= \{f \in \mathcal{A}\left(\operatorname{Iw}_{G}, \mathcal{O}_{\Omega}\right) :f\left(\overline{n} t g\right)=\chi_{\Omega}(t) f(g) \text { for all } \overline{n} \in \overline{N}\left(\mathbf{Z}_{p}\right) \cap \operatorname{Iw}_{G}, t \in T\left(\mathbf{Z}_{p}\right), g \in \mathrm{Iw}_{G} \}.$

	\end{definition}

	By \eqref{a} and the Iwahori decomposition, we can reduce the domain of functions in $\mathcal{A}_{\Omega}$ and simply consider functions in $\mathcal{A}\left(N\left(\mathbf{Z}_{p}\right), \mathcal{O}_{\Omega}\right)$. 
	
	Moreover, by dualising the above definition we get the following distribution spaces:

	Define $ \cD_{\Omega}=\operatorname{Hom_{cts}}\left( \cA_{\Omega},\cO_\Omega\right)$ and $\cD\left(\Zp^{\times}, \cO_\Omega\right)= \operatorname{Hom_{cts}}\left(\cA(\Zp^{\times}, \cO_\Omega),\cO_\Omega\right).$

	We will now define the spaces that will be used for the $p$-adic interpolation of the Branching Law (Theorem \ref{brl}).
 
 Let $
	\operatorname{Iw}_{G}^{1}:=\overline{N}\left(p \mathbf{Z}_{p}\right) \cdot T\left(\mathbf{Z}_{p}\right) \cdot N^{1}\left(\mathbf{Z}_{p}\right) \subset \operatorname{Iw}_{G},
	$ and $
	\operatorname{Iw}_{H}^{1}:=H\left(\mathbf{Z}_{p}\right) \cap  \operatorname{Iw}_{G}^{1}.$
	
	Similarly, $
	\operatorname{Iw}_{G}^{\beta}:=\overline{N}\left(p \mathbf{Z}_{p}\right) \cdot T\left(\mathbf{Z}_{p}\right) \cdot N^{\beta}\left(\mathbf{Z}_{p}\right).$ 
	
	\begin{definition}
		We say that $\phi \in \mathcal{A}_{\Omega}$  has support on $\operatorname{Iw}_{G}^{\beta}$ if $\phi(g)=0$ for $g \notin \operatorname{Iw}_{G}^{\beta}$ . Let $\mathcal{A}_{\Omega}^{\beta} \subset \mathcal{A}_{\Omega}$  be the subspace of functions supported on $\operatorname{Iw}_{G}^{\beta}$. Similarly, define $\mathcal{D}_{\Omega}^{\beta} \subset \mathcal{D}_{\Omega}$ as the space of distributions $\mu \in \mathcal{D}_{\Omega}$ such that $\mu(\phi)=\mu\left( \phi\mid_{Iw_{G}^{\beta}}\right)$ for all $\phi \in \mathcal{A}_{\Omega}$.
	\end{definition}
	
	By Lemma \ref{1+pzp}, the following map is well defined:
	\begin{align*}
		u_\Omega : &N^1(\Zp) \to \cO^{\times}_\Omega  \nonumber \\
		g& \mapsto \prod_{i=0}^{n-1}u_{i+1,i}(g)^{(\chi_{\Omega_i}-\chi_{\Omega_{n+1+i}})}  \prod_{i=1}^{n}v_{i,i}(g)^{(\chi_{\Omega_{n+1+i}}-\chi_{\Omega_{i}})} \cdot u_{n+1,n}(g)^{\chi_{\Omega_{n+1}}} .
	\end{align*}

	By \eqref{a}, we can extend this map to $u_\Omega : \operatorname{Iw}_{G}^1(\Zp) \to \cO_\Omega$ and we get a function in $\cA_{\Omega}$ with support on $\operatorname{Iw}_{G}^1(\Zp).$

	Therefore, the following is well defined:
	
	\begin{definition}
		Define the map $u_{\Omega} \in \cA_{\Omega}$, where if $g \in \operatorname{Iw}_{G}^1(\Zp)$
		
		\begin{align*}
		    u_\Omega(g)=\prod_{i=0}^{n-1}u_{i+1,i}(g)^{(\chi_{\Omega_i}-\chi_{\Omega_{n+1+i}})}  \prod_{i=1}^{n}v_{i,i}(g)^{(\chi_{\Omega_{n+1+i}}-\chi_{\Omega_{i}})} \cdot u_{n+1,n}(g)^{\chi_{\Omega_{n+1}}} 
		\end{align*}

		and $u_\Omega(g)=0$ otherwise.
		
	\end{definition}
 
	We can therefore define the map $\tilde{u}_{\Omega} : \cA(\Zp^{\times}, \cO_\Omega) \to \cA^1_{\Omega}$ by setting
	
	\begin{align*}
		\tilde{u}_\Omega(f)(g)=&\prod_{i=0}^{n-1}u_{i+1,i}(g)^{(\chi_{\Omega_i}-\chi_{\Omega_{n+1+i}})} f\left(\frac{1}{u_{i+1,i}(g)}\right)\\
        &\prod_{i=1}^{n}v_{i,i}(g)^{(\chi_{\Omega_{n+1+i}}-\chi_{\Omega_{i}})} f\left(v_{i,i}(g)\right)\cdot u_{n+1,n}(g)^{\chi_{\Omega_{n+1}}} 
	\end{align*}

	Note that the motivation behind this definition is that 
 
 $$\tilde{u}_{\ml}\left(z \mapsto z^j \right)(g)=u_{(\mu,\la+j)}(g).$$
	
	We thus get a map 
	$$\tilde{u}_{\Omega} : \cA(\Zp^{\times}, \cO_\Omega) \to \cA^1_{\Omega}.$$

	Dualising $\tilde{u}_\Omega $, we get a map 
	$$ \kappa_\Omega : \cD^1_{\Omega} \to \cD(\Zp^{\times}, \cO_\Omega),$$
	
	where $ \cD^1_{\Omega}$ is the dual of $ \cA^1_{\Omega}$ and $\cD(\Zp^{\times}, \cO_\Omega)$ is the dual of $\cA(\Zp^{\times}, \cO_\Omega)$.

	\begin{proposition} \label{blprop}
		For each classical $\ml \in \Omega$ such that $\mu^{\vee}$ interlaces $\la$ and each $j \in \cml$ the following diagram commutes:
		\begin{center}
			\begin{tikzcd}
				
				\cD_\Omega^1 \arrow{d}{sp_{\ml}} \arrow{r}{\kappa_\Omega}
				& \cD(\Zp^{\times}, \cO_\Omega) \arrow{d}{sp_{\ml}} \\
				\cD_{\ml}^1(L) \arrow{d}{\tilde{\iota}} \arrow{r}{\kappa_{\ml}}	& \cD(\Zp^{\times},L) \arrow{d}{\xi \mapsto \xi(z^j)} \\
				\mathcal{V}_{\ml}^{\vee}(L) \arrow{r}{ev_{u_{(\mu,\la+j)}}}	& L
			\end{tikzcd}
			
		\end{center}
		where $\tilde{\iota} : \cD_{\ml}(L) \to \mathcal{V}_{\ml}^{\vee}(L)$ is induced by the inclusion $\iota : 	\mathcal{V}_{\ml}(L) \hookrightarrow \cA_{\ml}(L).$
	\end{proposition}

	\begin{proof}
		
		The top square commutes since  $sp_{\ml} \circ \chi_\Omega= \ml$. For the bottom square,  first we observe that $j \in \cml$, $ev_{u_{(\mlj)}} \in \operatorname{Hom}_H\left(\mathcal{V}_{\mlj}^{\vee},L\right)=\operatorname{Hom}_H\left(\mathcal{V}_{\ml}^{\vee},\mathcal{V}_{j}\right)$ and $\mathcal{V}_{j}$ is the highest weight $H$-representation of weight $(j,...,j)$ which is one-dimensional. We can now see that the bottom square is well defined and we have
		$$
		\begin{aligned}
			\kappa_{\ml}(\xi)\left(z \mapsto z^j\right) &=\int_{\Zp^{\times}} z^j \cdot d \kappa_{\ml}(\xi) \\
            &=\int_{\operatorname{Iw}_{G}} \tilde{v}_{\ml}\left(z^j\right) \cdot d \xi \\ &=\int_{\operatorname{Iw}_{G}^{1}} \tilde{v}_{\ml}\left(z^j\right) \cdot d \xi\\
			&=\int_{\operatorname{Iw}_{G}^{1}} u_{(\mu,\la+j)} \cdot \xi \\
                &=\int_{\operatorname{Iw}_{G}} u_{(\mu,\la+j)} \cdot \xi \\
                &=\left(ev_{u_{(\mu,\la+j)}} \circ \tilde{\iota}\right)(\xi),  
		\end{aligned}
		$$
		
		since $\tilde{u}_{\ml}\left(z \mapsto z^j \right)(g)=u_{(\mu,\la+j)}(g).$ Therefore, the bottom square commutes as well. 
		
	\end{proof}
	
	Under the isomorphism $\Zp^{\times} \cong \Gac,$ we have that for a Dirichlet character $\chi$ and $j \in \cml$, the character $x^j\chi(x) \in \cA(\Zp^{\times}, L)$ corresponds to an accessible character $\widetilde{\chi_j} \in \cA(\Gac, L)$ of infinity type $(j,-j)$ and vice versa. Moreover, $\kappa_{\ml}(\xi) \in \cD(\Zp^{\times}, L)$ corresponds to $\widetilde{\kappa_{\ml}(\xi)} \in \cD(\Gac, L)=\operatorname{Hom_{cts}}\left( \cA(\Gac,L),L\right).$ We thus have
	
	\begin{equation}
		\int_{\Zp^{\times}}x^j\chi(x) d\kappa_{\ml}(\xi) = \int_{\Gamma^{\text{ac}}} \widetilde{\chi_j} d\widetilde{\kappa_{\ml}(\xi)}. 		\label{integral}
	\end{equation}
	
	Similarly, $\kappa_{\Omega}(\xi) \in \cD(\Zp^{\times}, \cO_\Omega)$ corresponds to $$\widetilde{\kappa_{\Omega}(\xi)} \in \cD(\Gac, \cO_\Omega)=\operatorname{Hom_{cts}}\left( \cA(\Gac,\cO_\Omega),\cO_\Omega\right).$$
	
	The following corollary now follows:
	
	\begin{corollary}\label{blcor}
		For each classical $\ml \in \Omega$ such that $\mu^{\vee}$ interlaces $\la$ and each $j \in \cml$ the following diagram commutes:
		\begin{center}
			\begin{tikzcd}
				
				\cD_\Omega^1 \arrow{d}{sp_{\ml}} \arrow{r}{\widetilde{\kappa_\Omega}}
				& \cD(\Gac, \cO_\Omega) \arrow{d}{sp_{\ml}} \\
				\cD_{\ml}^1(L) \arrow{r}{\widetilde{\kappa_{\ml}}}	& \cD(\Gac,L) .
			\end{tikzcd}
			
		\end{center}
	\end{corollary}

	\section{Coleman Families}\label{fam}
	
	Since for unitary groups every arithmetic subgroup is finite, we can define spaces of Gross' style algebraic modular forms for $G=\Uno\times\Un$ following \cite{gross}.

	\begin{definition}
		The following $\mathbb{Q}$-vector space is called the space of algebraic modular forms over $G$ with coefficients in an irreducible $\Uno \times \Un$-representation $V:$ 
		\begin{multline*}
		    M(G, V)=\{f: G(\A) \to V: f \text{ is locally constant on } G(\A) \text{ and }  f(\gamma g)=\gamma f(g) \text{ for all } \gamma \in G(\Q)\}.
		\end{multline*}
	
	\end{definition}

	The reason we would like to work with algebraic modular forms is because we can decompose $M(G, V)$ as the direct limit of the subspaces 
	$$
	M(G, V, K)=\{f: G(\mathbb{A})/(G(\mathbb{R})_{+}\times K) \to V: f(\gamma g)=\gamma f(g) \text{ for all } \gamma \in G(\mathbb{Q})\}
	$$
	for $K$ compact open in $G(\hat{\mathbb{Q}})$. That is because $f \in M(G,V)$ is locally constant and it is therefore constant on cosets of open subgroups of the form $G(\mathbb{R})_{+} \times K$. Moreover, $f \in  M(G, V,K)$ is determined by its values on the cosets of the double quotient $G(\mathbb{Q}) \backslash {G(\mathbb{A})}/(G(\mathbb{R})_{+}\times K).$
	
	In order to $p$-adically interpolate the automorphic periods, we would like to consider a $p$-adic version of the algebraic modular forms. To do that, we will use the following lemma:
	
	\begin{lemma}
		Assume that $K=K_p \times K'$, with $K'$ open and compact in $G(\hat{\mathbb{Q}}) \backslash G(\mathbb{Q}_p)$. Then there is an isomorphism 
		$$M(G, V,K)\otimes\mathbb{Q}_p \cong \{F: G(\mathbb{Q})\backslash G(\mathbb{A})/(G(\mathbb{R}_{+}) \times K') \to V\otimes\mathbb{Q}_p :F(gk_{p})=k_{p}^{-1}F(g) \text{ for all } k_{p} \text{ in } K_p\}.$$
	\end{lemma}
	
	\begin{proof}
		See \cite[Prop 8.6]{gross}
	\end{proof}
	
	We denote the right hand side of this isomorphism by $M_p(G,V,K)$ and we drop $K$ from the notation when the choice is clear.

     From now on, we consider level $K=\IwG\cdot K^{(p)}$ and thus drop this from notation, writing $M_p(G,V).$
	
	Consider the locally symmetric space for $G$ of level $K$
	$$Y_K=G(\Q) \backslash G(\A) / K.$$  This space has dimension zero since it's a finite set and therefore we have that $H_c^i(Y_K,*)=0$ for all $i>0$. We thus only study $H_c^0(Y_K,*)=M_p(G,*).$

	 Denote by $\widetilde{\ps_{\alpha}}$ the $p$-refinement of $\ps$ at a non-critical $\alpha$, in the sense of Definition \ref{crit}. 
	
	\begin{definition}\label{multassumption}
		We say that $\widetilde{\ps_{\alpha}}$ has the multiplicity one property with respect to an open compact subgroup $K$ of $G(\A_f)$ if the space of the $K$-invariants of $\widetilde{(\pi_f,\sigma_f)_{\alpha}}$ is 1-dimensional.
	\end{definition}
	
	We now fix a cohomological automorphic representation $\ps$ of $G$ of level $K$ and a non-critical $p$-refinement $\widetilde{\ps_{\alpha}}$ which satisfies the multiplicity one property with respect to its level $K$.
	
	\begin{remark} 
		The assumption of the existence of an open compact $K_l \subset G(\Q_l)$ such that $(\pi_l,\sigma_l)^{K_l}$ is 1-dimensional is needed due to the lack of a theory of newforms in the case of unitary groups. This is an active area of research and we justify this assumption by noting that this is known to hold at almost all places for odd unitary groups by the work of \cite{AOY}. In the even case, \cite{Atobe23} gives a partial answer. We only need this assumption for the variation in families of Theorem \ref{mainthm}, not for Theorem \ref{thm1}.
	\end{remark}
	
	 We consider an affinoid $\Omega$ around the weight $\ml \in \cW$. We will now consider the distribution-valued algebraic modular forms $M_p(G,\cD_{\Omega})$ which come equipped with a natural specialisation map to the classical ones. A Coleman family will be an element in this space interpolating the classical forms. 

  The following proposition guarantees the existence of a Coleman family through $\ps$.

  \begin{proposition}
 The space  $M_p(G,\cD_{\Omega})_{\widetilde{\ps_{\alpha}}}$ is free of rank 1 over $\cO_\Omega.$
      
  \end{proposition}

\begin{proof}

The point $\ml$ corresponds to a maximal ideal $\fm_{\ml}=\{f \mid f(\ml)=0\}\subset \cO_\Omega$ and we can consider:
	
	$$ M_p(G,\cD_{\Omega}) \otimes\cO_\Omega/\fm_{\ml} \hookrightarrow M_p(G,\cD_{\Omega}\otimes\cO_\Omega/\fm_{\ml}) \cong  M_p(G, \cD_{\ml}).$$
	
	The Hecke algebra $\He$ acts on this space and thus gives slope decompositions. Because slope decompositions are exact, we get
	
	\begin{equation}
		M_p(G, \cD_\Omega)^{\leq h} \otimes \cO_\Omega/\fm_{\ml}  \hookrightarrow M_p(G, \cD_{\ml})^{\leq h}. 		\label{eq1}
	\end{equation}
	
	By \cite[Thm 4.4]{BW21}, for non-critical slope $\alpha$, we have that distribution-valued forms of slope $\alpha$ are classical. This means that by localising at a non-critical slope refinement $\widetilde{\ps}_{\alpha}$, we have a Hecke-equivariant isomorphism 
	
	\begin{equation}
		M_p(G, \cD_{\ml})_{\widetilde{\ps_{\alpha}}} \tilde{\rightarrow} M_p(G, V^{\vee}_{\ml})_{\widetilde{\ps_{\alpha}}}. 		\label{eq2}
	\end{equation}

	By \cite[Lemma 2.9]{BDJ}, we can apply Nakayama's lemma on (\ref{eq1}) and since the right-hand side in (\ref{eq2}) is 1-dimensional over $\cO_\Omega$ by the multiplicity one property, we get that $M_p(G,\cD_{\Omega})_{\widetilde{\ps_{\alpha}}}$ is free of rank 1 over $\cO_\Omega$.
    
\end{proof} 

    We therefore fix a generator of $M_p(G, \cD_{\ml})_{\widetilde{\ps_{\alpha}}}$ which lifts to a generator of $M_p(G,\cD_{\Omega})_{\widetilde{\ps_{\alpha}}}$, which in turn by \cite[Lemma 2.10]{BDJ} can be lifted to an eigenclass $\Phi_{\Omega} \in M_p(G,\cD_{\Omega})$, up to shrinking $\Omega.$
      
      \begin{definition}
          An eigenclass $\Phi_{\Omega} \in M_p(G,\cD_{\Omega})$ arising from the above procedure is called a Coleman family through ${\widetilde{\ps_{\alpha}}}$.
      \end{definition}

	\section{The Gan--Gross--Prasad conjecture}\label{ggp}

The aim of this section is to exhibit the link between automorphic period integrals twisted by anticyclotomic characters and central $L$-values. We use results from \cite[Section 2]{Liu}.
	
	For an accessible character $\chi: \Gac \to \bar{L}^{\times}$ of infinity type $(-j,j)$, for some $j \in \cml$ and conductor $p^{\beta}$ with $\beta>0$ and $\phi \in\ps$, we define the automorphic $(H, \chi)$-period integral 
	
	$$
	\mathscr{P}_{H, \chi}(\phi) \coloneqq \int_{[H]} \phi(h) \chi\left(\operatorname{det}(h)\right) d h. 
	$$ 
	where we choose the Haar measure on $GL_n(\Qp)$ that gives a hyperspecial maximal subgroup volume 1. We also fix a decomposition $dh = dh^{\infty}\cdot dh^p \cdot dh^{\infty,p}$ of the Tamagawa measure on $H(\A)$ such that the volume of $H(\R)$ under $dh^{\infty}$ is 1.
	
	Consider the following integral operator over matrix coefficients 
	
	$$\alpha^{\chi}_{\nu} \in \operatorname{Hom}_{H(\mathbb{Q}_{\nu})\times H(\mathbb{Q}_{\nu})}\left(\ps_{\nu,\chi}\times\ps^{\vee}_{\nu,\chi}, \bar{L} \right)$$
	
	defined by

	$$\alpha^{\chi}_{\nu} \left(\phi_{\nu},\phi^{\vee}_{\nu}\right)\coloneqq\int_{H(\mathbb{Q}_{\nu})}\langle \pi_{\nu}(h_{\nu})\phi_{\nu}, \phi^{\vee}_{\nu}\rangle_{\nu}\chi(\operatorname{det}(h_{\nu})) dh_{\nu}.$$
	
	As explained in \cite[Lemma 2.6]{Liu}, we have that for an embedding $\iota: \bar{L} \to \C$,
	
	$$\iota\left(\int_{H(\mathbb{Q}_{\nu})}\langle \pi_{\nu}(h_{\nu})\phi_{\nu}, \phi^{\vee}_{\nu}\rangle_{\nu}\chi(\operatorname{det}(h_{\nu})) dh_{\nu}\right)=\int_{H(\mathbb{Q}_{\nu})}\iota\left(\langle \pi_{\nu}(h_{\nu})\phi_{\nu}, \phi^{\vee}_{\nu}\rangle_{\nu}\chi(\operatorname{det}(h_{\nu})) \right) dh_{\nu}.$$
	
	We will therefore drop $\iota$ from the notation and consider the $\alpha^{\chi}_{
		\nu}$'s as taking values in $\C.$
	
	Similarly, by  \cite[Lemma 2.4]{Liu}, we have that the local $L$-factor at $p$ behaves nicely under $\iota$, that is
	
	$$\iota\left(L_{\nu}(1/2, \pi\times\left(\sigma\otimes\chi\right))\right)=L_{\nu}(1/2, \iota\left(\pi\right)\times\iota\left(\sigma\otimes\chi\right)).$$
	
	To link the automorphic period integral with $L$-values, we will use the unitary Gan--Gross--Prasad conjecture, which is now a theorem. For the interested reader we recommend \cite{ggp} for the initial statement of the conjecture and we note that the following theorem is stated and proved in its full generality in  \cite[Theorem 1.1.6.1]{GGPfinal}, see also \cite{BPLZZ}.
	
	\begin{theorem}[The Gan--Gross--Prasad Conjecture]\label{ggpthm}
		With notation as above, we have that
		$$\frac{\mid\iota\left(\cP_{H, \chi}(\phi)\right)\mid^2}{\langle \phi, \phi^{\vee} \rangle}=\frac{\Delta}{S_{\pi\times\sigma}}\frac{L^S(1/2, \pi\times\left(\sigma\otimes\chi\right))}{L^S(1,\pi,\mathrm{Ad})L^S(1,\sigma,\mathrm{Ad})}\prod_{\nu \in S}\alpha^{\chi}_{\nu}(\phi_{\nu},\phi^{\vee}_{\nu}),$$
		
		where $\Delta_{n+1}=\prod_{i=1}^{n+1} L\left(i, \eta^i\right)$ and $S_{\pi\times\sigma}$ is the size of the Arthur parameter of $\pi\times\sigma$ and $S$ is a large enough set of primes containing $p$.
		
	\end{theorem}
	
	We would actually like to apply the above formula in a twisted setting where we evaluate the automorphic period integral at $$\tilde{\phi_{\beta}}=\otimes_{\nu \neq p}\phi_{\nu}\otimes\pi\left(\left( \xi, 1_n\right) \cdot t_p^{\beta}\right) \phi_{p}$$ and similarly for $\tilde{\phi}_{\beta}^\vee$, where 
	
	$\xi=\begin{pmatrix}
		& w_{n} &	&  &\begin{matrix} 1 \\ 1 \\ \vdots \end{matrix} \\ 
		0   & \cdots && 0 &1 \\
	\end{pmatrix}.$ 
	
	\begin{proposition}\label{testvectors}
		There exist $\phi$, $\phi^{\vee}$ such that for each accessible character $\chi: \Gamma \to \bar{L}^{\times}$ of conductor $p^{\beta}$ with $\beta>0$:
		
		\begin{align}\frac{\mid\cP_{H, \chi}(\tilde{\phi}_{\beta})\mid^2}{\langle \tilde{\phi}_{\beta}, \tilde{\phi}_{\beta}^\vee\rangle} =&\frac{\Delta}{S_{\pi\times\sigma}} \frac{L(1/2, \pi\times\left(\sigma\otimes\chi\right))}{L(1,\pi,\mathrm{Ad})L(1,\sigma,\mathrm{Ad})} \nonumber \\
			& c\prod_{i=1}^n\left(1-p^{-i}\right)^{-2} \cdot\left(p^{-\frac{n(n+1)(2n+1)}{6}}\right)^{\beta} \left(i_1 \phi\right)(1) \left(i_2 \phi^{\vee}\right)(1). \nonumber 	
		\end{align}
		
	\end{proposition}
	
	\begin{proof}
		The proof proceeds as in \cite{Liu}, we give a sketch for the convenience of the reader.
		
		The choice of $\phi$, $\phi^{\vee}$ is as follows:
		
		\begin{itemize}
			\item At $p$, we choose an Iwahori-fixed vector, which exists by the assumption that $\ps$ is Iwahori spherical.
			\item Away from $p$, we choose vectors that satisfy the conclusion of \cite[Prop. 2.10]{Liu}. In particular, this guarantees the appropriate cancellations in the Gan--Gross--Prasad formula.
		\end{itemize}
		We now perform some local calculations at $p$, so we supress $p$ from the notation:
		
		One can view $\alpha^{\chi}$ as a function on the Whittaker functionals by choosing $G(\Qp)$-equivariant isomorphisms
		$i_1: \pi \otimes_{\mathbb{L}, \iota} \mathbb{C} \stackrel{\sim}{\rightarrow} \mathcal{W}(\iota \pi)_\psi$ and $i_2 \in \pi^{\vee} \otimes_{\mathbb{L}, \iota} \mathbb{C} \stackrel{\sim}{\rightarrow} \mathcal{W}\left(\iota \pi^{\vee}\right)_{\psi^{-1}},$ where $\psi$ is the usual additive character.
		
		By \cite[Prop 4.10]{Zha14} we have that for every ramified character $\chi: E^{\times,-} \rightarrow \mathbb{C}^{\times}$ and for every $\phi \in \pi$ and $\phi^{\vee} \in \pi^{\vee}$ 
		
		\begin{align}
			\alpha^{\chi}\left(\phi, \phi^{\vee}\right)=&c\left(\int_{(U \cap H)(\Qp) \backslash H(\Qp)}\left(i_1 \phi\right)(h) \cdot \chi(\operatorname{\operatorname{det}} h)  \mathrm{d} h\right) \nonumber \\	
			&\left(\int_{(U \cap H)(\Qp) \backslash H(\Qp)}\left(i_2 \phi^{\vee}\right)(h) \cdot \chi(\operatorname{\operatorname{det}} h)^{-1}  \mathrm{~d} h\right), \nonumber
		\end{align}
		
		where $c$ is a constant depending only on the choice of measures.
		
		We can now apply the above formula to $\pi\left(\left(\xi, 1_n \right) \cdot t_p^{\beta}\right) \phi_p$ and $\pi^{\vee}\left(\left(\xi, 1_n\right) \cdot t_p^{\beta}\right) \phi_p^{\vee}$ when $\chi$ has conductor $p^{\beta}$. We have assumed that $G$ splits at $p$ and we can therefore apply the Local Birch Lemma of \cite[Cor.2.8]{Jan11}, which gives 
		$$
		\begin{aligned}
			& \alpha^{\chi}\left(\pi\left(\left(\xi, 1_n \right) \cdot t_p^{\beta}\right) \phi_p, \pi^{\vee}\left(\left(\xi, 1_n\right) t_p^{\beta}\right) \phi_p^{\vee}\right) \\
			& =c \prod_{i=1}^n\left(1-p^{-i}\right)^{-2} \cdot\left(p^{-\frac{n(n+1)(n+2)}{3}}\right)^{\beta} \cdot G_{\psi_F}(\chi)^{\frac{n(n+1)}{2}} G_{\psi_F^{-1}}\left(\chi^{-1}\right)^{\frac{n(n+1)}{2}} \cdot\left(i_1 \phi\right)(1) \cdot\left(i_2 \phi^{\vee}\right)(1).
		\end{aligned}
		$$

		A quick calculation shows that $G_{\psi_F}(\chi)G_{\psi_F^{-1}}\left(\chi^{-1}\right)=p^{\beta}.$
		
		Returning to the global setting, for an accessible character $\chi$ of $\Gamma^{\text{ac}}$ of conductor $p^{\beta}$ with $\beta>0$, we let $\tilde{\phi}_{\beta}=\otimes_{\nu \neq p}\phi_{\nu}\otimes\pi\left(\left(\xi, 1_n\right) \cdot t_p^{\beta}\right) \phi_{p}$ and similarly for $\tilde{\phi}_{\beta}^\vee,$. It then follows by Theorem \ref{ggpthm} above and our specific choice of vectors that:

		\begin{align}	
			\frac{\mid\cP_{H, \chi}(\tilde{\phi}_{\beta})\mid^2}{\langle \tilde{\phi}_{\beta}, \tilde{\phi}_{\beta}^\vee\rangle}=&\frac{\Delta}{S_{\pi\times\sigma}} \frac{L(1/2, \pi\times\left(\sigma\otimes\chi\right))}{L(1,\pi,\mathrm{Ad})L(1,\sigma,\mathrm{Ad})} \nonumber \\
			&c \prod_{i=1}^n\left(1-p^{-i}\right)^{-2} \cdot\left(p^{-\frac{n(n+1)(2n+1)}{6}}\right)^{\beta} \left(i_1 \phi\right)(1) \left(i_2 \phi^{\vee}\right)(1). \nonumber
		\end{align}
	\end{proof}
	
	\section{The \texorpdfstring{$p$}{p}-adic \texorpdfstring{$L$}{L}-function}\label{plf}
 
 We are now ready to prove our main results. From now on, we assume that $\ps$ has multiplicity one property with respect to the level $K$.
	
	To ease notation, we set $u=(\xi, 1_n)$. We abuse notation and denote by $M_p(-)$ the localised spaces $M_p(-)_{\widetilde{\ps_{\alpha}}}.$ We will be lowering the level $K$,  to a level $K_{\beta}=(u t_p^{\beta}) K (u t_p^{\beta})^{-1},$ where $\beta>0.$ It is easy to see (e.g \cite[Lemma 4.11]{Liu}) that $\left[K_{\beta} : K_{\beta+1}\right]=p^{\frac{n(n+1)(n+2)}{3}}$ and that for $\phi \in M_p(G,\cD_{\Omega})$, $\tilde{\phi_{\beta}}=\otimes_{\nu \neq p}\phi_{\nu}\otimes\pi\left(u \cdot t_p^{\beta}\right) \phi_{p} \in M_p(G,\cD^{\beta}_{\Omega}).$
	
	For ${\Phi_\Omega} \in M_p(G,\cD^{\beta}_{\Omega}),$ $\varPhi \in M_p(G,\cD^{\beta}_{\ml})$ and $\phi \in M_p(G,\mathcal{V}_{\ml}^{\vee}),$ we define respectively:
	
	$$\mathfrak{L}_{p, \beta}({\Phi_\Omega}) \coloneqq \mu(K_{\beta}) \sum_{x \in [H]/K_{\beta}}
	\frac{\widetilde{\kappa_{\Omega}}({\Phi_\Omega}(x))}{\left|Z^{H}(\A)H(\Q) \cap xK_{\beta}x^{-1}\right|} \in \cD(\Gac, \cO_\Omega),$$
	
	$$\cL_{p, \beta}(\varPhi) \coloneqq \mu(K_{\beta}) \sum_{x \in [H]/K_{\beta}}
	\frac{\widetilde{\kappa_{\ml}}(\varPhi(x))}{\left|Z^{H}(\A)H(\Q) \cap xK_{\beta}x^{-1}\right|} \in \cD(\Gac, L),$$
	
	$$L_p(\phi) \coloneqq \mu(K) \sum_{x \in [H]/K_{\beta}}
	\frac{ev_{\ml}(\phi(x))}{\left|Z^{H}(\A)H(\Q) \cap xKx^{-1}\right|} \in L,$$
	
	and we note that the last definition is just the automorphic $H$-period $\cP_{H}(ev_{\ml}\circ\phi ).$
	
	It is clear that $L_p(\phi)=\alpha_{p}^{-\beta} L_p(\tilde{\phi_{\beta}})$ is independent of $\beta$ for all $\beta>0.$

	By combining Proposition \ref{blprop} and Corollary \ref{blprop} with the results of Section \ref{fam}, we have:
	
	\begin{proposition}
		
		For all accessible characters $\widetilde{\chi_j}$ of infinity type $(-j,j)$ for $j \in \cml$ and conductor $p^{\beta}$ the following diagram commutes
		
		\begin{center}\label{blfinal}
			\begin{tikzcd}
				
				M_p(G,\cD^{\beta}_{\Omega}) \arrow{d}{sp_{\ml}} \arrow{r}{\mathfrak{L}_{p, \beta}}
				& \cD(\Gac, \cO_\Omega) \arrow{d}{sp_{\ml}} \\
				M_p(G,\cD^{\beta}_{\ml}) \arrow{d}{\theta} \arrow{r}{\cL_{p, \beta}}	& \cD(\Gac,L) \arrow{d}{\xi \mapsto \xi(\widetilde{\chi_j})} \\
				M_p(G,\mathcal{V}_{\ml}^{\vee}) \arrow{r}{\cP_{H, \widetilde{\chi_j}}\circ ev_{\mlj}}	& L.
			\end{tikzcd}
			
		\end{center}
		
	\end{proposition}
	
	\begin{proof}
		By Corollary \ref{blcor} we have that the top square commutes since $$sp_{\ml}\left(\mathfrak{L}_{p, \beta}({\Phi_\Omega})\right)=\cL_{p, \beta}(sp_{\ml}\left({\Phi}\right)).$$
		
		For the bottom square, first note that we can rewrite $\cP_{H,\widetilde{\chi_j}}$ so that it's suitable for $p$-adic interpolation. We have for $\phi$ of level $K_{\beta}$
		
		$$
		\cP_{H, \widetilde{\chi_j}}(\phi)= \mu(K_{\beta})  \sum\limits_{x \in [H]/K_{\beta}} \frac{\widetilde{\chi_j}(\operatorname{det}(x)){ev_{\mlj}(\phi(x))}}{\left|Z^{H}(\A)H(\Q) \cap xK_{\beta}x^{-1}\right|}. 
		$$
		
		The bottom square thus commutes by combining Proposition \ref{blprop} with equation (\ref{integral}), which give
		
		$$\cP_{H, \widetilde{\chi_j}}({\theta}(\varPhi))= \int_{\Gamma^{\text{ac}}} \widetilde{\chi_j} d\cL_{p, \beta}(\varPhi).$$

	\end{proof}
For $\varPhi \in M_p(G,\cD_{\ml})$ with $U_p$-eigenvalue $a_p$, define $$\cL(\varPhi)\coloneqq a_p^{-\beta}\cL_{p,\beta}\tilde{\varPhi_\beta}.$$

The rest of this section aims to explain how this is a $p$-adic $L$-function. One of the key steps is the following proposition:
    
    \begin{proposition}\label{growth}
        For an eigenfunction ${\varPhi} \in M_p(G,\cD_{\ml})$ with $U_p$-eigenvalue $a_p$, the distribution $$\cL_{p}({\varPhi})\in \cD(\Gac, L)$$ has growth $u_p(a_p).$
    \end{proposition}

    \begin{proof}

  As explained in Section \ref{distributions}, $\cD(\Gac, L) \subset \cD(\Gac, L)[r]$ for all $r$ and thus comes equipped with a norm $\lvert \cdot \rvert_r$ as a Banach space.
  We want to show that $\cL_{p}({\Phi})$ has growth $h=u_p(a_p),$ which is equivalent to showing that for all there exists constant $C$ such that for all $r$

  $$\lvert\cL_{p}({\Phi})\rvert_{r} \leq Cp^{rh}.$$

  We have 
\begin{align*}
      \lvert\cL_{p}({\varPhi})\rvert_{r}=\lvert a_p^{-r}\cL_{p,r}(\tilde{\varPhi_r})\rvert_{r} \\
      =p^{rh}\lvert \cL_{p,r}(\tilde{\varPhi_r})\rvert_{r}.
  \end{align*}

     Since

     $$\cL_{p, r}(\tilde{\varPhi}_r) \coloneqq \mu(K_{r}) \sum_{x \in [H]/K_{r}}
	\frac{\widetilde{\kappa_{\ml}}(\tilde{\varPhi}_r(x))}{\left|Z^{H}(\A)H(\Q) \cap xK_{r}x^{-1}\right|} \in \cD(\Gac, L),$$

it is enough to show that there exists constant $C$ independent of $r$ such that for all $x$:

    $$\lvert\widetilde{\kappa_{\ml}}(\tilde{{\varPhi}}_r(x))\rvert_r \leq C,$$ 

    or equivalently

    $$\lvert{\kappa_{\ml}}(\tilde{{\varPhi}_r}(x))\rvert_r \leq C.$$ 

    We now have:
\begin{equation*}
\begin{aligned}
    \lvert{\kappa_{\ml}}(\tilde{\varPhi}_r(x))\rvert_r= &\sup_{f \in \cA(\Zp^{\times}, L)[r]} \frac{ \lvert{\kappa_{\ml}}(\tilde{\varPhi}_r(x))(f)\rvert}{\lvert f \rvert_r} \\
    =&\sup_{\substack{f \in \cA(\Zp^{\times}, L)[r]\\ \lvert f\rvert_r=1}} \lvert{\kappa_{\ml}}(\tilde{\varPhi}_r(x))(f)\rvert\\
    =&\sup_{\substack{f \in \cA(\Zp^{\times}, L)[r]\\ \lvert f\rvert_r=1}}\lvert\tilde{\varPhi}_r(x)(\tilde{u}_{\ml}(f))\rvert\\
    \leq & \sup_{\substack{g \in \cA_{\ml}[r]\\ \lvert g\rvert_r=1}}\lvert\tilde{\varPhi}_r(x)(g)\rvert \\
    \leq & \sup_{\substack{g \in \cA_{\ml}[1]\\ \lvert g\rvert_1=1}}\lvert\tilde{\varPhi}_r(x)(g)\rvert \\
    = &\lvert \tilde{\varPhi}_r(x)\rvert_1 \\
    \leq & \lvert \varPhi (x) \rvert_1 \leq C,\\
    \end{aligned}
\end{equation*}
    
where the last inequality follows since $\varPhi (x)$ can take finitely many values, so we take $C=\max_{x}\{\lvert \varPhi (x) \rvert_1\}.$

    \end{proof}

	We will now assume that the basis for the 1-dimensional space $\ps_f^K$ is an appropriate test vector, as in Proposition \ref{testvectors}. 
	
	\begin{remark}\label{remarktv}
		This is not known in generality. Recent work of \cite[Corollary 1.3]{Dang} shows that the multiplicity one assumption implies that the $K_l$-fixed vector is a test vector, under the additional assumption that the representation $\sigma_l$ of $U_n((\Q_l)$ is unramified.
	\end{remark}

	We are now ready to prove Theorem \ref{mainthm} of the introduction. 
	\begin{theorem}\label{Mthm}
		There exists a unique up to scalars non-zero function
		$$\mathfrak{L}_p : M_p(G,\cD_{\Omega})\to \cD(\Gamma^{\text{ac}}, \cO_\Omega)$$ 	 
		of growth $v_p(\alpha),$ such that for a classical point $\ml \in \Omega$ whose corresponding automorphic representation $\ps$ satisfies ($\bigstar$), integrating $sp_{\ml}\left(\mathfrak{L}_p(\Phi_{\Omega})\right)$ against an accessible character $\chi \in \cX(\Gamma^{\text{ac}})$ of conductor $p^{\beta}$ with $\beta>0$, interpolates the square root of the central critical value $L(1/2, \pi\times(\sigma \otimes \chi)).$ In particular, there exists test vector $\phi$, such that if we set
		
		$$\cL_p(\phi)=sp_{\ml}\left(\mathfrak{L}_p(\Phi_{\Omega})\right) \text{  and  
 }\cL_p(\widetilde{\ps}, \chi) \coloneqq \int_{\Gamma^{\text{ac}}} \chi(x) d\cL_{p}(\phi),$$ then
		
		\begin{align}
			\frac{|\cL_p (\widetilde{\ps}, \chi)|^2}{\langle \tilde{\phi}_{\beta}, \tilde{\phi}_{\beta}^\vee\rangle}=&\frac{\Delta}{S_{\pi\times\sigma}} \frac{L(1/2, \pi\times\left(\sigma\otimes\chi\right))}{L(1,\pi,\mathrm{Ad})L(1,\sigma,\mathrm{Ad})}\nonumber\\
			&c \prod_{i=1}^n\left(1-p^{-i}\right)^{-2} \cdot\left(\frac{	p^{-\frac{n(n+1)(2n+1)}{6}}}{\alpha_p}
			\right)^{\beta} \left(i_1 \phi\right)(1) \left(i_2 \phi^{\vee}\right)(1).\nonumber
		\end{align}

	\end{theorem}
	
	\begin{proof}

		Consider an element $\phi \in M\left(G, \mathcal{V}_{\ml}^{\vee}\right)$, such that $\phi \circ ev_{u_{(\mu,\la)}} \in \ps$ is a test vector as in the proof of Proposition \ref{testvectors}. We know by Section \ref{ggp} and the above  assumption on the test vectors that
		
		\begin{align}
			\frac{|\cP_{H, \chi}(\tilde{\phi}_{\beta} \circ ev_{u_{(\mu,\la)}})|^2}{\langle \tilde{\phi}_{\beta}, \tilde{\phi}_{\beta}^\vee\rangle}=&\frac{\Delta}{S_{\pi\times\sigma}} \frac{L(1/2, \pi\times\left(\sigma\otimes\chi\right))}{L(1,\pi,\mathrm{Ad})L(1,\sigma,\mathrm{Ad})} \nonumber \\
			&c \prod_{i=1}^n\left(1-p^{-i}\right)^{-2} \cdot\left(p^{-\-\frac{n(n+1)(2n+1)}{6}}\right)^{\beta}\left(i_1 \phi\right)(1) \left(i_2 \phi^{\vee}\right)(1). \nonumber
		\end{align}
		
		Define $\mathfrak{L}_{p} : M_p(G, \cD_{\Omega})^{\alpha} \to \cD(\Gamma^{\text{ac}}, \cO_\Omega)$ via
		$$\mathfrak{L}_{p}({\Phi})=\mathfrak{L}_{p, \beta}((\alpha_p)^{-\beta}\tilde{\Phi}_{\beta})=\frac{\mu(K_{\beta}) }{\alpha_{p}^{\beta}}\sum_{x \in [H]/K_{\beta}}  \frac{\widetilde{\kappa_{\Omega}}(\tilde{\Phi}_{\beta} (x))}{\left|Z^{H}(\A)H(\Q)  \cap x^{-1}K_{\beta} x\right|},$$
		
		where $\tilde{\Phi}_{\beta}=\otimes_{\nu \neq p}\Phi_{\nu}\otimes\pi\left(\left(\xi, 1_n\right) \cdot t_p^{\beta}\right) \Phi_{p}$. This definition is independent of the choice of $\beta$ and is well defined by the following lemma:
  \begin{claim}
       The element $(g_0, 1_n)$ belongs to the same $\bar{B}\backslash G /H$-orbit as $(\xi,1_n).$
    \end{claim}

  \begin{proof}
      	For an integer $k\geq 1$, let $w_k$ denote the $k \times k$ matrix with 1s in the anti-diagonal entries and zeros everywhere else.
		A simple calculation shows that the element $g_0=\begin{psmallmatrix}
			1 & 0 & \cdots &0 &1 \\
			0 & 1 & \cdots &0 &1 \\
			\vdots  & \vdots  &  & \ddots &\vdots \\
			0 & 0 & \cdots & 0 &1 \\
		\end{psmallmatrix}$ lies in the same $\overline{B}\backslash G /  H $-orbit as $\xi=\begin{pmatrix}
			& w_{n} &	&  &\begin{matrix} 1 \\ 1 \\ \vdots \end{matrix} \\ 
			0 & \cdots && 0 &1 \\
		\end{pmatrix}$ since 
		
		$$\begin{psmallmatrix}
			1 & 0 & \cdots &0 &1 \\
			0 & 1 & \cdots &0 &1 \\
			\vdots  & \vdots  &  & \ddots &\vdots \\
			0 & 0 & \cdots & 0 &1 \\
		\end{psmallmatrix} \times \begin{pmatrix}
			& w_{n} &	&  &\begin{matrix} 0 \\ 0 \\ \vdots \end{matrix} \\ 
			0 & \cdots && 0 &1 \\
		\end{pmatrix} = \begin{pmatrix}
			& w_{n} &	&  &\begin{matrix} 1 \\ 1 \\ \vdots \end{matrix} \\ 
			0 & \cdots && 0 &1 \\
		\end{pmatrix}.$$
  \end{proof}

		Recall that in the beginning of the proof, we fixed a test vector $\phi \in M_p(G,\mathcal{V}_{\ml}^{\vee})$ as in the proof of Proposition \ref{testvectors}. By the isomorphism $\theta$, this lifts to $\varPhi \in M_p(G,\cD_{\ml})$ and by Section \ref{fam}, we have a Coleman family $\PO$ over it. 
		
		For an accessible character  $\chi \in \cX(\Gamma^{\text{ac}})$ of conductor $p^{\beta}$ with $\beta>0$ of infinity type $(-j,j)$, we can integrate $\chi(x)$ against $\cL_{p}(\varPhi)=sp_{\ml} \mathfrak{L}_p(\PO),$ for $sp_{\ml}(\PO)=\varPhi \in  M_p(G, \cD_{\ml})$. Therefore, we set
		
		\begin{align}
			\cL_p(\widetilde{\ps}, \chi) & \coloneqq \int_{\Gamma^{\text{ac}}} \chi(x) d\cL_{p}(\varPhi)\nonumber \\
			&=(\alpha_p)^{-\beta} \int_{\Gamma^{\text{ac}}} \chi(x) d\cL_{p, \beta}(\tilde{\varPhi_{\beta}}). \nonumber
		\end{align}

		Finally, we apply Proposition \ref{blfinal} and Proposition \ref{testvectors} to get
		
		\begin{align}
			\frac{|\cL_p (\widetilde{\ps}, \chi)|^2}{\langle \tilde{\phi}_{\beta}, \tilde{\phi}_{\beta}^\vee\rangle}=&\frac{\Delta}{S_{\pi\times\sigma}} \frac{L(1/2, \pi\times\left(\sigma\otimes\chi\right))}{L(1,\pi,\mathrm{Ad})L(1,\sigma,\mathrm{Ad})} \nonumber \\
			&c \prod_{i=1}^n\left(1-p^{-i}\right)^{-2} \cdot\left(\frac{	p^{-\frac{n(n+1)(2n+1)}{6}}}{\alpha_p}\right)^{\beta} \left(i_1 \phi\right)(1) \left(i_2 \phi^{\vee}\right)(1).\nonumber
		\end{align}

		The growth condition of Proposition \ref{growth} and the Zariski density of classical points satisfying $v_p(\alpha)<h_{\ml}$ (as explained for example in \cite{Hansen}) determine the $p$-adic $L$-function uniquely as shown in  \cite[Thm 2.3 and Lemma 2.10]{Vishik} (also proved independently in \cite{AV}). 		
	\end{proof}
	
	We now remove the multiplicity one condition on the level and we prove Theorem \ref{thm1}, as a direct consequence of the construction outlined above:
	
	\begin{theorem}\label{theorem1}
		There exists a distribution $\cL_p(\widetilde{\ps}) \in \cD(\Gamma^{\text{ac}}, L)$ of growth $v_p(\alpha)$ such that for all accessible characters $\chi \in \cX(\Gamma^{\text{ac}})$ of conductor $p^{\beta}$ with $\beta>0$, if we set
		
		$$\cL_p(\widetilde{\ps}, \chi) \coloneqq \int_{\Gamma^{\text{ac}}} \chi(x) d\cL_{p}(\widetilde{\ps}),$$ 
		
		then 
		
		\begin{align}
			\frac{|\cL_p (\widetilde{\ps}, \chi)|^2}{\langle \tilde{\phi}_{\beta}, \tilde{\phi}_{\beta}^\vee\rangle}=&\frac{\Delta}{S_{\pi\times\sigma}} \frac{L(1/2, \pi\times\left(\sigma\otimes\chi\right))}{L(1,\pi,\mathrm{Ad})L(1,\sigma,\mathrm{Ad})}\nonumber\\
			&c \prod_{i=1}^n\left(1-p^{-i}\right)^{-2} \cdot\left(\frac{	p^{-\frac{n(n+1)(2n+1)}{6}}}{\alpha_p}
			\right)^{\beta} \left(i_1 \phi\right)(1) \left(i_2 \phi^{\vee}\right)(1).\nonumber
		\end{align}

		If moreover, $v_p(\alpha)< h_{\ml}$, then $\cL_p(\widetilde{\ps})$ is uniquely determined by the above interpolation formula.
		
	\end{theorem}
	
	\begin{proof}
		The proof follows directly from the construction of the $p$-adic $L$-function $\cL_p$ and the arguments in the proof of the previous theorem. In particular, we no longer require the multiplicity one assumption on the level but we only get uniqueness if $v_p(\alpha)< h_{\ml}$.
	\end{proof}

	\bibliographystyle{alpha}
	\bibliography{XDI}

\end{document}